\documentclass{amsart}
\usepackage[dvips]{epsfig}
\usepackage{graphicx}
\usepackage{latexsym}
\usepackage{amsmath}
\usepackage{amsthm}
\usepackage{amssymb}
\usepackage{natbib}
\setcitestyle{numbers,square}
\RequirePackage[colorlinks=true,citecolor=blue,urlcolor=blue]{hyperref}

\usepackage{xcolor}

\usepackage{caption}
\usepackage{subcaption}
\usepackage{float}

 \usepackage[applemac]{inputenc}

\newtheorem{lemma}{Lemma}
\newtheorem{example}{Example}
\newtheorem{corollary}{Corollary}
\newtheorem{definition}{Definition}

\newtheorem{proposition}{Proposition}
\newtheorem{remark}{Remark}

\newenvironment{preuve}{\vip \noindent {\it Proof}}{\hfill$\square$\vip}

\newcommand{\vip}{\vskip.2cm}

\newcommand{\COMMENTAIRE}[1]{}

\newcommand\numberthis{\addtocounter{equation}{1}\tag{\theequation}}

\def\restriction#1#2{\mathchoice
              {\setbox1\hbox{${\displaystyle #1}_{\scriptstyle #2}$}
              \restrictionaux{#1}{#2}}
              {\setbox1\hbox{${\textstyle #1}_{\scriptstyle #2}$}
              \restrictionaux{#1}{#2}}
              {\setbox1\hbox{${\scriptstyle #1}_{\scriptscriptstyle #2}$}
              \restrictionaux{#1}{#2}}
              {\setbox1\hbox{${\scriptscriptstyle #1}_{\scriptscriptstyle #2}$}
              \restrictionaux{#1}{#2}}}
\def\restrictionaux#1#2{{#1\,\smash{\vrule height .8\ht1 depth .85\dp1}}_{\,#2}}

\begin{document}

\title[Control the difference between two Brownian motions with dynamic copulae]{On the control of the difference between two Brownian motions: a dynamic copula approach}

\author{Thomas Deschatre}

\address{Thomas Deschatre, CEREMADE,
Universit\'e Paris-Dauphine, Place du mar\'echal De Lattre de Tassigny
75775 Paris Cedex 16, France.}

\email{thomas.deschatre@gmail.com}

\begin{abstract} We propose new copulae to model the dependence between two Brownian motions and to control the distribution of their difference. Our approach is based on the copula between the Brownian motion and its reflection. We show that the class of admissible copulae for the Brownian motions are not limited to the class of Gaussian copulae and that it also contains asymmetric copulae. These copulae allow for the survival function of the difference between two Brownian motions to have higher value in the right tail than in the Gaussian copula case. Considering two Brownian motions $B_t^1$ and $B_t^2$, the main result is that the range of possible values for $\mathbb{P}\left(B_t^1-B^2_t \geq \eta\right)$ with $\eta > 0$ is the same for Markovian pairs and all pairs of Brownian motions, that is $\left[0,2\Phi\left(\frac{-\eta}{2\sqrt{t}}\right)\right]$ with $\Phi$ being the cumulative distribution function of a standard Gaussian random variable.  
\end{abstract}

\maketitle




\section{Introduction}

\subsection{Motivation}

Modeling dependence between risks has become an important problem in insurance and finance. An important application in risk management for commodity energy markets is the pricing of multi-asset options, and in particular the pricing of spread options.  Spread options are used to model the incomes of a plant, such as coal plant. A review on the spread options and on the pricing and hedging models is done by Carmona \citep{carmona03}. The simplest model used for derivative pricing and hedging on several underlying is the multivariate Black and Scholes model \citep{carmona05}. Each price is modeled by a geometric Brownian motion and the dependence between the different Brownian motions is modeled by a constant correlation matrix. The copula between the Brownian motions when they are linked by correlation is called a Gaussian copula. Copulae have many applications in finance and insurance, especially in credit derivative modeling. For instance, Li \citep{li99} used the Gaussian copula to model the dependence between time until default of different financial instruments. For more information on the use of copulae in finance, the reader can refer to \citep{cherubini04}.

\medskip
Let $X_t$ be the price of electricity at time $t$, $Y_t$ the price of coal and $H$ the heat rate (conversion factor) between the two. The income of the coal plant at time $t$ can be modeled by $(X_t - H Y_t - K)^+$ where $K$ is a constant and corresponds to a fixed cost (we have neglected the price of carbon emissions). Coal is a combustible used to produce electricity and $H$ is the cost of one unit of coal used to produce one unit of electricity. Thus we expect to have $X_t > H Y_t$, i.e. the price of electricity greater than the price of the coal used to produce it, with a probability greater than $\frac{1}{2}$. Let us consider that the two commodities are modeled by an arithmetic Brownian motion with a zero drift under a risk neutral probability $\mathbb{P}$: $X_t = \sigma^X B^1_t$ and $H Y_t = \sigma^Y B^2_t$ and we suppose that $\langle dB^1, dB^2\rangle_t = \rho dt$. The dependence between the two Brownian motions is modeled by a correlation, i.e a Gaussian copula. For $x \in \mathbb{R}$, we have 
\[\mathbb{P}\left( X_t - H Y_t \geq x\right) =  \mathbb{P}\left( X_t - H Y_t \leq -x\right)\]
and then, if $x \geq 0$, 
\[\mathbb{P}\left( X_t - H Y_t \geq x\right) \leq \frac{1}{2}.\]
The distribution of the difference between the two prices is symmetric and moreover, the value of its survival function is limited to $\frac{1}{2}$ in the right tail. We would like to have higher values for this probability in order to enrich our modeling. The modeling of the dependence with a constant correlation does not allow to capture the asymmetry in the distribution of the difference of the prices and limits the values that can be achieved by its survival function. Today, it is common practice to use a factorial model \citep{benth08} to model prices of commodities which is based on Brownian motions. Marginal models, i.e. when we consider only one commodity at the time, are enough performant for risk management. However, the dependence between them is modeled by a Gaussian copula, which is not enough to capture the asymmetry and the values taken by the survival function of their difference. Sklar's Theorem \citep{sklar59} states that the structure of dependence can be separated from the modeling of the marginals with the copula. Studying the impact of the structure of dependence on the modeling is equivalent to studying the impact of the copula.

\medskip
Whereas copulae are very useful in a static framework where random variables are modeled, modeling with copulae is much more difficult in a dynamic framework, that is when processes are involved. In a discrete time framework, Patton \citep{patton06} introduces the conditional copula which is a copula at time $t$ defined conditionally on the information at time $t-1$. Fermanian and Wegkamp \citep{fermanian04} generalize the concept of conditional copula. In a continuous time framework, Darsow et al. \citep{darsow92} consider the modeling of the time dependence by a copula. They give sufficient and necessary conditions for a copula to be the copula of a Markov process $X = \left(X_t\right)_{t \geq 0}$ between times $t$ and $s$, i.e. the copula of $\left(X_t, X_s\right)$, using the Chapman-Kolmogorov equation. We are more interested in the space dependence, that is the dependence between two different processes at a given time $t$. The question is studied by Jaworski and Krzywda \citep{jaworski13}. They consider two Brownian motions and they are interested in copulae that make the bivariate process self-similar. They find necessary and sufficient conditions for the copula to be suitable for the Brownian motions deriving the Kolmogorov forward equation. The copula is linked by a local correlation function into a partial derivative equation. Further work has been done in the thesis of Bosc \citep{bosc12} where there are no constraints of self-similarity and it is not only limited to Brownian motions ; a more general partial derivative equation is found. More details about their work are given in Section \ref{browniancopula}. However, conditions for the copula to be suitable for the Brownian motions are very restrictive. An equivalent approach to the copula one is the coupling approach. A coupling of two stochastic processes is a bi-dimensional measure on the product space such that the marginal measures correspond to the ones of the stochastic processes. For more information on coupling, the reader can refer to \citep{chen89}. One of the most important coupling is the coupling by reflection \citep{lindvall86}, based on the reflection of the Brownian motion which has some importance in this article.

\subsection{Objectives and results}

The objective of this article is to control the distribution of the difference between two Brownian motions at a given time $t$. The distribution of the difference between two Brownian motions $B^1$ and $B^2$ can be described by $x \mapsto \mathbb{P}\left(B^1_t - B^2_t \geq x \right)$, $x \in \mathbb{R}$, $t \geq 0$. If $B^1_t-B^2_t$ has a continuous cumulative distribution function, this function is the survival function of $B^1_t - B^2_t$ at point $x$. In particular, we want to find asymmetric distributions for $B^1-B^2$ with more weight in the positive part than in the Gaussian copula case, i.e. $\mathbb{P}\left(B^1_t - B^2_t \geq \eta \right)$ greater than $\frac{1}{2}$ for a given $\eta > 0$. Since distributions of $B^1_t$ and $B^2_t$ are known, we control this distribution with the copula of $\left(B^1_t,B^2_t\right)$. One of the main issues is to work in a dynamical framework ; we then first need to extend the definition of copulae to Markovian diffusions. If we denote by $\mathcal{C}_B$ the set of admissible copulae for Brownian motions, which is properly defined in Section \ref{browniancopula}, our main goal is to study the range of the function 
\[\begin{array}{ccccc}
S_{\eta,t} & : & \mathcal{C}_B & \to & \left[0,1\right] \\
 & & C & \mapsto & \mathbb{P}_{C}\left(B^1_t - B^2_t \geq \eta\right) \\
\end{array}
\]
denoted by $Ran\left(S_{\eta,t}\right)$ with $\mathbb{P}_C$ the probability measure associated to $\left(B^1,B^2\right)$ when $C \in \mathcal{C}_B$ and with $\eta > 0$ and $t \geq 0$ given. Our problem is related to the one consisting in finding bounds of $\mathbb{P}\left(X + Y < \eta\right)$ with $X$ and $Y$ two random variables with given distributions, see \cite[Section 6.1]{nelsen06} and \citep{frank87, ruschendorf82, makarov82}. In our case, we add the constraint that the copulae are in $\mathcal{C}_{B}$.  

\medskip
Considering the set of Gaussian copulae, it is easy to prove that $\left[0,\Phi\left(\frac{-\eta}{2\sqrt{t}}\right)\right] \subset Ran\left(S_{\eta,t}\right)$ by controlling the correlation between the two Brownian motions with $\Phi$ the cumulative distribution function of a standard normal random variable. Furthermore, if we consider the restriction of $S_{\eta,t}$ to the set of Gaussian copulae $\restriction{S_{\eta,t}}{C^d_G}$, we have $Ran\left(\restriction{S_{\eta,t}}{C^d_G}\right) =  \Bigl[0,\Phi\left(\frac{-\eta}{2\sqrt{t}}\right)\Bigr]$, see Proposition \ref{boundcopuladynamic} (i) below.

\medskip
Our major contribution is to construct a family of dynamic copulae in $\mathcal{C}_B$ that can achieve all the values between 0 and the supremum of $S_{\eta,t}$ on $\mathcal{C}_B$.  We first prove that  
\[\underset{C \in \mathcal{C}_B}{\sup \;}  S_{\eta,t}\left(C\right) =  2\Phi\Bigl(\frac{-\eta}{2\sqrt{t}}\Bigr)\]
in Proposition \ref{boundcopuladynamic} (ii), implying that the Gaussian copulae can not describe all the values that can be achieved by $S_{\eta,t}$. This supremum is achieved with the copula of the Brownian motion and its reflection, which we call the Reflection Brownian Copula, and which a closed formula is given in Proposition \ref{copulareflection}. In the case where $X$ and $Y$ are two normal random variables $\mathcal{N}\left(0,\sqrt{t}\right)$, the supremum of $\mathbb{P}\left(X-Y \geq \eta \right)$ without constraint on the copula is also equal to $2\Phi\Bigl(\frac{-\eta}{2\sqrt{t}}\Bigr)$, see Proposition \ref{boundcopula} (ii) below that follows from \cite[Section 6.1, Example 6.2]{nelsen06}. The supremum for $\mathbb{P}_{C}\left(B^1_t - B^2_t \geq \eta\right)$ is then the same for Markovian pairs and all pairs of Brownian motions. Deriving a new family of copulae that is described in Proposition \ref{randombarrier} from the Reflection Brownian Copula, it is possible to achieve all the value between 0 and $2\Phi\Bigl(\frac{-\eta}{2\sqrt{t}}\Bigr)$, which means that 
\[Ran\left(S_{\eta,t}\right) =  \Bigl[0,2\Phi\Bigl(\frac{-\eta}{2\sqrt{t}}\Bigr)\Bigr];\]
this is the result of Proposition \ref{boundcopuladynamic} (iii). The range of possible values for $\mathbb{P}_{C}\left(B^1_t - B^2_t \geq \eta\right)$ is the same for Markovian pairs and all pairs of Brownian motions. Copulae used to achieve values in $Ran\left(S_{\eta,t}\right)$ present two states depending on the value of $B^1_t - B^2_t$: one of positive correlation and one of negative one. These copulae are asymmetric and to our knowledge, these are the only asymmetric copulae suitable for Brownian motions available in the literature.  

\subsection{Structure of the paper}

In Section \ref{markovcopulae}, we define the notion of dynamic copulae for Markovian diffusion processes and in particular for the case of two Brownian motions. We show that our definition includes several model of dependence present in the literature such as stochastic correlation models. In Section \ref{reflectionbrowniancopula}, we compute a copula called the Reflection Brownian Copula based on the dependence between a Brownian motion and its reflection and we derive new families of asymmetric copulae based on this copula. In Section \ref{spreaddistributioncontrol},  after showing the limitations of modeling the dependence between two random variables with symmetric copulae, we establish the results on the range of the function $S_{\eta,t}$, first in a static framework and then in a dynamical framework with Brownian motions.

\section{Markov Diffusion Copulae}
\label{markovcopulae}

In finance and insurance, modeling of two dimensional processes is usually based on a 2 dimensional Brownian motion, that is when the structure of dependence between two 1 dimensional Brownian motions is modeled by a correlation. The copula of the two Brownian motions at a given time then belongs to the class of Gaussian copulae.

\medskip
Let us recall that a function $C: \left[0,1\right]^2 \mapsto \left[0,1\right]$ is a copula if:
\begin{enumerate}
\item[(i)] $C$ is 2-increasing, i.e. $C\left(u_2,v_2\right) - C\left(u_1,v_2\right) + C\left(u_1,v_1\right) - C\left(u_2,v_1\right) \geq 0 \text{ for }  u_2 \geq u_1, v_2 \geq v_1$ and $u_1, u_2, v_1, v_2 \in \left[0,1\right]$,
\item[(ii)]  $C\left(u,0\right) = C\left(0,v\right) = 0$, $u, v \in \left[0,1\right]$,
\item[(iii)]  $C\left(u,1\right) = u, C\left(1,u\right) = u$,  $u \in \left[0,1\right]$.
\end{enumerate}
We denote by $\mathcal{C}$ the set of copulae and by $\mathcal{C}_G$ the set of Gaussian copulae completed by the upper and lower Frechet copulae $M\left(u,v\right) = \min\left(u,v\right)$ and $W\left(u,v\right) = \max \left(u+v-1,0\right)$ corresponding to the limit cases $\rho = 1$ and $\rho = -1$. $\mathcal{C}_G = \{ C \in \mathcal{C} : \exists \rho \in \left(-1,1\right), C = C_{G,\rho}\} \cup \{M,W\}$ where $C_{G,\rho}$ denote the Gaussian copula with parameter $\rho$. We have \[C_{G,\rho}(u,v) = \Phi_{\rho}(\Phi^{-1}(u),\Phi^{-1}(v))\] with $\Phi$ the cumulative distribution function of a standard normal random variable and $\Phi_\rho$ the cumulative distribution function of a bivariate normal random variable with correlation $\rho$:
\[\Phi_{\rho}(x,y) = \int_{-\infty}^{y} \int_{-\infty}^{x} \frac{1}{2 \pi \sqrt{1-\rho^2}}e^{-\frac{1}{2(1-\rho^2)}(u^2+v^2-2\rho u v)}du dv.\]
In the following, a Gaussian copula will abusively refer to an element of $\mathcal{C}_G$.

\medskip
In this section, we want to generalize the concept of copula which is adapted for random variables to a dynamical framework. We want to define the notion of copula for Markov diffusions in Section \ref{defmarkovcopula}. In particular, we are interested in copulae suitable for Brownian motions in Section \ref{browniancopula}.

\subsection{Definition}
\label{defmarkovcopula}

\medskip
In order to work in a dynamical framework, we need to extend the concept of copula to Markovian diffusions. Our definition is based on the work of Bielecki et al. \citep{bielecki08} and gives a more general definition.

\medskip
We recall that if $P = \left(P_t\right)_{t \geq 0}$ is a Markovian diffusion solution of the stochastic differential equation 
\[dP_{t} = \mu\left(P_t\right)dt + \sigma\left(P_{t}\right) dW_t,\]
with $W = \left(W_t\right)_{t \geq 0}$ a standard Brownian motion, the infinitesimal generator $\mathcal{L}$ of $P$ is the operator defined by  
\[\mathcal{L}f\left(x\right) = \frac{1}{2}\sigma^2\left(x\right) f''\left(x\right) + \mu\left(x\right)f'\left(x\right)\]
for $f$ in a suitable space of functions including $\mathcal{C}^2$.

\begin{definition}[Admissible copula for Markovian diffusions] \label{admissiblecopula} We say that a collection of copula $C = \left(C_t\right)_{t \geq 0}$ is an admissible copula for the $n$ real valued Markovian diffusions, $n \geq 2$, $\left(X^i\right)_{ 1 \leq i \leq n }$ defined on a common probability space $\left(\Omega, \mathcal{F}, \mathbb{P}\right)$ if there exists a $\mathbb{R}^m$ Markovian diffusion $Z = \left(Z^i\right)_{1 \leq i \leq m}$, $m \geq n$, defined on a probability extension of $\left(\Omega, \mathcal{F}, \mathbb{P} \right)$ such that 
\[
\left\lbrace
\begin{array}{c}
\mathcal{L}\left(Z^i\right) = \mathcal{L}\left(X^i\right), 1 \leq i \leq n,\\
Z^i_0 = X^i_0, 1 \leq i \leq n, \\
\text{for } t \geq 0, \text{ the copula of } \left(Z_t^i\right)_{1 \leq i \leq n} \text{ is } C_t.
\end{array}\right.
\]
\end{definition}

\medskip
The strongest constraint to be admissible is that $Z$ has to be a Markovian diffusion. Without this constraint, all the copulae are admissible. Sempi \citep{sempi10} studies the Brownian motions linked by a copula without this constraint. Definition \ref{admissiblecopula} is consistent with the approach of \citep{jaworski13} or \citep{bosc12} consisting of modeling dependence by a local correlation function. However, our approach is totally different. 

\subsection{Brownian motion case}
\label{browniancopula}

\medskip
From now on, we work in a 2 dimensional framework and we denote by $\mathcal{C}_B$ the set of admissible copulae for Brownian motions, that is when $X^1$ and $X^2$ are Brownian motions. The only well known suitable copulae for Brownian motion are the Gaussian copulae.

\medskip
We can extend the definition of $\mathcal{C}_G$ to a dynamical framework by defining 
\[\mathcal{C}^d_G = \{(C_t)_{t \geq 0} : \; \forall t \in \mathbb{R}^+, \; C_t \in \mathcal{C}_G\} \cap \mathcal{C}_B.\]
It is necessary to take the intersection with $\mathcal{C}_B$ because we do not know if conditions are needed on $C_t$ for the copula to be admissible. We are not interested in this question in this paper. However, we know this intersection is not empty because $\{(C_t)_{t \geq 0} : \exists \rho \in \left(-1,1\right), \; \forall t \in \mathbb{R}^+ \; C_t = C_{G,\rho}\} \subset \mathcal{C}_B$. One of our objective is to find copulae that are admissible for Brownian motion but that are not Gaussian copulae.

\medskip
Jaworski and Krzywda \citep{jaworski13} prove that the set of admissible copulae for Brownian motions was not reduced to the Gaussian copulae. By linking local correlation and copula with the Kolmogorov backward equation, they find that a sufficient condition to be admissible is 
\begin{equation}
\label{pdecopula}
\left| \frac{1}{2} e^{\frac{\Phi^{-1}\left(v\right)^2-\Phi^{-1}\left(u\right)^2}{2}}  \frac{\partial^{2}_{u,u} C\left(u,v\right)}{\partial^{2}_{u,v} C\left(u,v\right)} + \frac{1}{2} e^{\frac{\Phi^{-1}\left(u\right)^2-\Phi^{-1}\left(v\right)^2}{2}}  \frac{\partial^{2}_{v,v} C\left(u,v\right)}{\partial^{2}_{u,v} C\left(u,v\right)} \right| < 1 \quad \forall (t,u,v) \in \mathbb{R^+} \times \left[0,1\right]^2
 \end{equation}
when the copula does not depend on time. In particular, they prove that the extension of the FGM copula $C^{FGM}\left(u,v\right) = uv\left(1+a\left(1-u\right)\left(1-v\right)\right), a \in \left[-1, 1\right]$ in a dynamical framework defined by $C_t\left(u,v\right) = C^{FGM}\left(u,v\right), \; t \geq 0$, is an admissible copula for Brownian motions. Bosc \citep{bosc12} has also found admissible copulae.

\bigskip
Let us consider two independent Brownian motions $B^1$ and $Z$ defined on a common probability space $\left(\Omega, \mathcal{F}, \mathbb{P}\right)$. Definition \ref{admissiblecopula} includes several models for Brownian motions used in the literature.

\medskip
\paragraph{ \it Deterministic correlation}  Let us consider a function $t \mapsto \rho\left(t\right)$ defined on $\mathbb{R}^+$ with values in $\left[-1,1\right]$.
\smallskip
Let $B_t^2 = \int_0^t \rho\left(s\right) dB_s^1 + \int_0^t \sqrt{1-\rho\left(s\right)^2} dZ_s$.

\smallskip
$B^2$ is a Brownian motion and the dynamic copula defined at each time $t$ by the copula of $\left(B^1_t,B^2_t\right)$ is in $\mathcal{C}_B$.

\medskip
\paragraph{\it Local correlation} Let us consider a function $\left(x,y\right) \mapsto \rho\left(x,y\right)$ defined on $\mathbb{R}^+$ with values in $\left[-1,1\right]$ and measurable. If the stochastic differential equation 
\[dB^2_s = \rho\left(B^1_t,B^2_t\right)dB^1_s + \sqrt{1-\rho\left(B^1_t,B^2_t\right)^2}dZ_s\]
has a strong solution, the dynamic copula defined at each time $t$ by the copula of $\left(B^1_t,B^2_t\right)$ is in $\mathcal{C}_B$ by the L\'evy characterization of Brownian motion.

\medskip
\paragraph{\it Stochastic correlation} Let us consider a Markovian diffusion $\rho = \left(\rho_s\right)_{s \geq 0}$ independent of $\left(B^1,Z\right)$ locally square integrable and with values in $\left[-1,1\right]$. 

\smallskip
We can extend the probability space and the filtration generated by $\left(B^1,Z\right)$. The stochastic process $B^2$ defined by $B_t^2 = \int_0^t \rho\left(s\right) dB_s^1 + \int_0^t \sqrt{1-\rho\left(s\right)^2} dZ_s$ is a Brownian motion and the dynamic copula defined at each time $t$ by the copula of $\left(B^1_t,B^2_t\right)$ is in $\mathcal{C}_B$.

\smallskip
We can also consider a correlation diffusion driven by $B^1$, $Z$ and an independent Brownian motion. If the system of stochastic differential equations has a strong solution, the copula is still in $\mathcal{C}_B$.

\medskip
Contrary to the approaches of Jaworski and Krzywda \citep{jaworski13}, Bosc \citep{bosc12} or Bielecki et al. \citep{bielecki08}, Definition \ref{admissiblecopula} includes stochastic correlation models. However, we need for the stochastic correlation to be a Markovian diffusion which is not needed in a general case ; the stochastic correlation has only to be progressively measurable.

\section{Reflection Brownian Copula}
\label{reflectionbrowniancopula}

In this section, our objective is to construct Markov Diffusion Copulae defined in Section \ref{markovcopulae}. We construct a new copula based on the reflection of the Brownian motion. We show that the copula between the Brownian motion and its reflection is adapted to a dynamical framework and is a suitable copula for Brownian motions. Furthermore, we give a closed formula of this copula in Section \ref{closedformula}. To our knowledge, this copula has not been studied in detail and it is the new copula suitable for Brownian motions. We also construct new families of copulae by extension of the Reflection Brownian Copula in Section \ref{extensions}. 

\subsection{Closed formula for the copula}
\label{closedformula}

\medskip
In this section, we study the copula between the Brownian motion and its reflection. Since its reflection is also a Brownian motion, the copula is a good candidate for being in $\mathcal{C}_B$. 
 
\bigskip
Let us consider a filtered probability space ($\Omega$, $\mathcal{F}$, $\left(\mathcal{F}_t\right)_{t \geq 0}$, $\mathbb{P}$) with $\left(\mathcal{F}_t\right)_{t \geq 0}$ satisfying the usual hypothesis (right continuity and completion) and $B = \left(B_t\right)_{t \geq 0}$ a Brownian motion adapted to $\left(\mathcal{F}_t\right)_{t \geq 0}$. We denote by $\tilde{B}^{h}$ the Brownian motion reflection of $B$ on $x = h$ with $h \in \mathbb{R}$, i.e. $\tilde{B}^{h}_t =  - B_t + 2 (B_t - B_{\tau^h}){\bf1}_{t \geq \tau^h}$ with $\tau^h =  \inf \{t  \geq 0 : B_t = h \}$. Thus, $\tilde{B}^{h}$ is a $\mathcal{F}$ Brownian motion according to the reflection principle (see \cite[Theorem\ 3.1.1.2, p.\ 137]{jeanblanc09}). Proposition \ref{copulareflection} gives the copula of $\bigl(B, \tilde{B}^{h}\bigr)$.

\medskip
We recall that $M\left(u,v\right) = \min\left(u,v\right)$, $W\left(u,v\right) = \max\left(u + v - 1, 0\right)$, $u, v \in \left[0,1\right]$ and that $\Phi$ denotes the cumulative distribution function of a standard normal random variable.

\medskip
\begin{proposition} \label{copulareflection} Let $h > 0$. The copula of $(B, \tilde{B}^{h})$, $\left(C^{ref,h}_t\right)_{t \geq 0}$, is defined by 
\begin{equation}
\label{copulareflectioneq}
C^{ref,h}_t\left(u,v\right) =
\left\lbrace
\begin{array}{ccc}
v & \mbox{if} & \Phi^{-1}\left(u\right) - \Phi^{-1}\left(v\right) \geq \frac{2h}{\sqrt{t}}  \\
W\left(u,v\right) + \Phi\left(\Phi^{-1}\left(M\left(u,1-v\right)\right) - \frac{2h}{\sqrt{t}}\right) & \mbox{if} & \Phi^{-1}\left(u\right) - \Phi^{-1}\left(v\right) < \frac{2h}{\sqrt{t}}
\end{array}
\right.
\end{equation}
and $\left(C^{ref,h}_t\right)_{t \geq 0} \in \mathcal{C}_B$. We call this copula the Reflection Brownian Copula. 
\end{proposition}

\subsection{Extensions}
\label{extensions}
\medskip

In this section, we give methods to construct new admissible copulae for Brownian motions from the Reflection Brownian Copula.

\medskip
Proposition \ref{nondegenerated} and its proof gives an approach to construct different admissible copulae for Brownian motions based on the Reflection Brownian Copula considering a correlated Brownian motion to the reflection of the Brownian motion. Thus, the copula of Proposition \ref{nondegenerated} is the copula of $\left(B, \rho \tilde{B}^h + \sqrt{1-\rho^2} Z\right)$, with $h > 0$, $\rho \in \left(0,1\right)$ and $Z = \left(Z_t\right)_{t \geq 0}$ a Brownian motion independent from $B$.

\begin{proposition}  \label{nondegenerated}
Let $h > 0$ and $\rho \in \left(0,1\right)$. The copula 
\[C_t(u,v) = \left\lbrace
\begin{array}{ccc}
\Phi_{\rho}\Bigl(\Phi^{-1}\left(u\right), \Phi^{-1}\left(v\right)+\frac{2\rho h}{\sqrt{t}}\Bigr) + v - \Phi\Bigl(\Phi^{-1}\left(v\right)+\frac{2\rho h}{\sqrt{t}}\Bigr) &\hspace{-0.1em} \mbox{if} & u  \geq \Phi\Bigl(\frac{h}{\sqrt{t}}\Bigr) \\
 \Phi_{-\rho}\Bigl(\Phi^{-1}\left(u\right),\Phi^{-1}\left(v\right)\Bigr) + \Phi_{\rho}\Bigl(\Phi^{-1}\left(u\right) - \frac{2h}{\sqrt{t}}, \Phi^{-1}\left(1-v\right)-\frac{2\rho h}{\sqrt{t}}\Bigr) +\\
   \Phi_{\rho}\Bigl(\Phi^{-1}\left(u\right)-\frac{2h}{\sqrt{t}}, \Phi^{-1}\left(v\right)\Bigr) -\Phi\Bigl(\Phi^{-1}\left(u\right)-\frac{2h}{\sqrt{t}}\Bigr)&\hspace{-0.1em} \mbox{if} & u  < \Phi\Bigl(\frac{h}{\sqrt{t}}\Bigr)
\end{array}\right.
\]

is in $\mathcal{C}^B$.
\end{proposition}

\medskip
Contrary to the Reflection Brownian Copula, this copula is non degenerated in the sense that we have two distinct sources of randomness. Indeed, in the Reflection Brownian Copula case, if we know the trajectory of the Brownian motion, we also know the one of its reflection.

\begin{remark} In the case $\rho = 0$, we still have a copula which is the independent copula and then that is in $\mathcal{C}_{B}$.
\end{remark}

\medskip
An other way to construct admissible copulae is to consider a random barrier $\xi$. By enlarging the filtration, the copula of the two processes is an admissible copula and it can be computed by integrating the copula of the Reflection Brownian motion according to the law of the barrier. Proposition \ref{randombarrier} gives the copula of $\bigl(B, \tilde{B}^{\xi}\bigr)$ which is clearly in $\mathcal{C}^{B}$ because $\bigl(B, \tilde{B}^{\xi}, \xi\bigr)$ is Markovian. 
\medskip

\begin{proposition} \label{randombarrier}Let $\xi$ be a positive random variable with law having a density and $\overline{F}^{\xi}$ its survival function. The copula 
\[
C^{\xi}_t\left(u,v\right) = v - \int_{-\infty}^{\Phi^{-1}\left(M\left(1-u, v\right)\right)} \frac{e^{\frac{-w^2}{2}}}{\sqrt{2\pi}}\overline{F}^{\xi}\Bigl(\frac{\sqrt{t}}{2}\left(\Phi^{-1}\left(M\left(u,1-v\right)\right) - w\right)\Bigr)dw 
\]
is in $\mathcal{C}^B$. 
\end{proposition}

\medskip
Example \ref{exponentialbarrier} below gives a copula with closed formula built with the method of Proposition \ref{randombarrier}. 

\begin{example} \label{exponentialbarrier} Let $\xi \overset{d}{=} h + X$ with $h \in \mathbb{R}$ and $X$ a random variable following an exponential law with parameter $\lambda > 0$. We have
$
\overline{F}^{\xi}(x) = \left\lbrace
\begin{array}{ccc}
1 & \mbox{if} & x \leq h \\
e^{-\lambda (x-h)} & \mbox{if} & x > h 
\end{array}\right.
$
and the copula 
\begin{align*}
\label{exponentialbarriereq}
&C^{exp,h,\lambda}_t\left(u,v\right) =W\left(u,v\right) + \min\Bigl[\Phi\Bigl(\Phi^{-1}\left(M\left(1-u,v\right)\right) - \frac{2h}{\sqrt{t}}\Bigr), M\left(u,1-v\right)\Bigr] \numberthis\\
&- \Phi\Bigl(\min\Bigl[\Phi^{-1}\left(M\left(1-u,v\right)\right) - \frac{2h}{\sqrt{t}}, \Phi^{-1}\left(M\left(u,1-v\right)\right)\Bigr] - \frac{\lambda\sqrt{t}}{2}\Bigr)e^{\lambda h + \frac{\lambda^2t}{4} + \frac{\lambda\sqrt{t}}{2}\Phi^{-1}\left(M\left(u,1-v\right)\right)}
\end{align*}
is in $\mathcal{C}_B$.
\end{example}

\medskip
The methods of Proposition \ref{nondegenerated} and \ref{randombarrier} could be used simultaneously to construct new classes of admissible copulae. Figure \ref{RBC} represents the Reflection Brownian Copula and some of its extensions.

\begin{figure}[h!]
    \centering
    \begin{subfigure}[b]{0.32\textwidth}
        \centering
        \includegraphics[width=\textwidth]{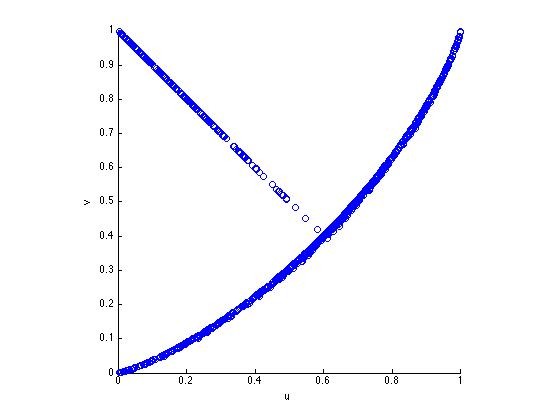}
        \caption{\label{RBCa}}
    \end{subfigure}
    \begin{subfigure}[b]{0.32\textwidth}
        \centering
        \includegraphics[width=\textwidth]{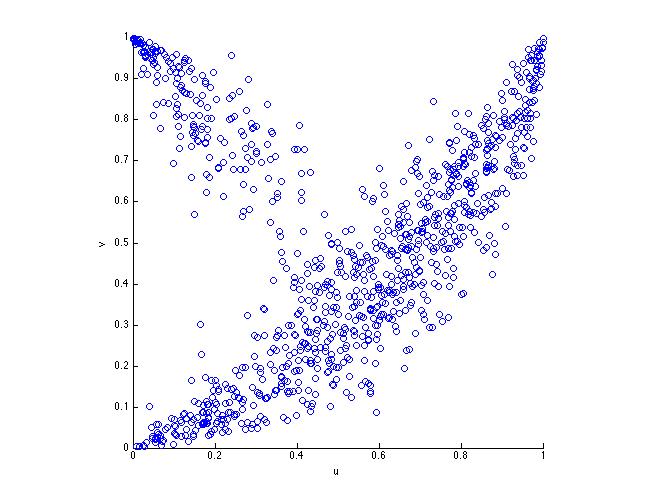}
        \caption{\label{RBCb}}
   
    \end{subfigure}
        \begin{subfigure}[b]{0.32\textwidth}
        \centering
        \includegraphics[width=\textwidth]{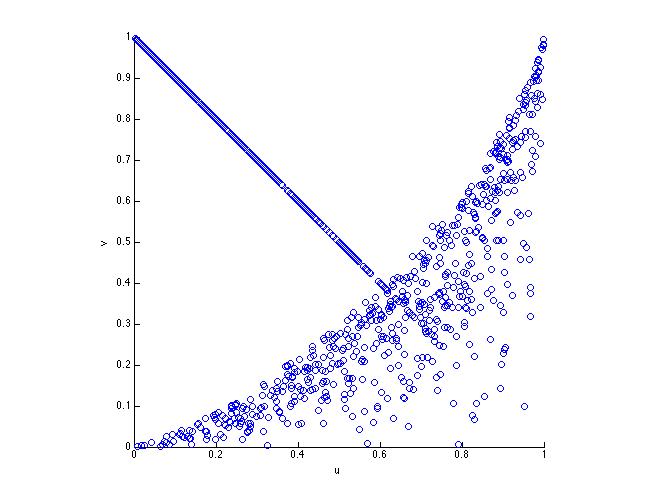}
        \caption{\label{RBCc}}
    \end{subfigure}
     \caption{\label{RBC} \it The Reflection Brownian Copula $C^{ref,h}$ and some of its extensions at time $t = 1$ with $h = 2$. Figure \ref{RBCa} is the Reflection Brownian Copula. Figure \ref{RBCb} is the extension considering a Brownian motion correlated to the reflection of the first Brownian with a correlation $\rho = 0.95$, which is the copula of Proposition \ref{nondegenerated}. Figure \ref{RBCc} is the extension in the case of a random barrier following an exponential law with parameter $\lambda = 2$, which is the copula of Example \ref{exponentialbarrier}.
     }
\end{figure}

\begin{remark} In Proposition \ref{copulareflection}, Proposition \ref{nondegenerated} and Proposition \ref{randombarrier}, the given copula is the copula between two Brownian motions $B_t^1$ and $B_t^2$ at time $t$. In Patton \citep{patton06}, the dynamic copula in a discrete time setting is the copula between $B_t^1$ and $B_t^2$ knowing the filtration generated by $\left(B_{t-1}^1, B_{t-1}^2, B_{t-2}^1, B_{t-2}^2, ..., B^1_0, B^2_0\right)$. In our case, the copula of $\left(B_t^1, B_t^2\right)$ is conditioned to the values of $B_0^1$ and $B_0^2$, which are equals almost surely to 0. One could compute the copula of $\left(B_t^1, B_t^2\right)$ knowing $\left(B_s^1, B_s^2\right)$, but this is not the subject of this paper. We want to study $\mathbb{P}\left(B_t^1-B_t^2 \geq \eta\right)$ that only requires the copula of $\left(B_t^1, B_t^2\right)$ and not $\mathbb{P}\left(B_t^1-B_t^2 \geq \eta\mid B_s^1, B_s^2\right)$.  Furthermore, the dependence between $B_t^1,B_t^2,B_s^1$ and $B_s^2$ is implicitely given in the dynamic of $\left(B_1, B_2\right)$ that we know. For application, we want to price European spread options of the form $\left(f\left(B_t^1\right) - g\left(B_t^2\right)\right)^+$, that depends only of the law of $\left(B_t^1, B_t^2\right)$ and hence of the copula of $\left(B_t^1, B_t^2\right)$. For American spread options, we know the dynamic of $\left(B^1, B^2\right)$ and we work in a Markovian framework. 
\end{remark}

\begin{remark}
The study of the range of $\mathbb{P}\left(B_t^1-B_t^2 \geq \eta\mid B_s^1, B_s^2\right)$ seems very close to our problem: $\mathbb{P}\left(B_t^1-B_t^2 \geq \eta\mid B_s^1, B_s^2\right) = \mathbb{P}\left(B_t^1-B_s^1-\left(B_t^2-B_s^2\right) \geq \eta - \left(B_s^1-B_s^2\right) \mid B_1^s, B_2^s\right)$ and $B_1^t-B_1^s$ (resp. $B_2^t-B_2^s$) is a Brownian motion independent from $B_1^s$ (resp. $B_2^s$). However, we let this problem for further studies.
\end{remark}
\section{Control of the distribution of the difference between two Brownian motions}
\label{spreaddistributioncontrol}

Let $B^1$ and $B^2$ be two standard Brownian motions defined on a common filtered probability space ($\Omega$, $\mathcal{F}$, $\left(\mathcal{F}_t\right)_{t \geq 0}$, $\mathbb{P}_C$) with $\left(\mathcal{F}_t\right)_{t\geq0}$ satisfying the usual hypothesis and where $\mathbb{P}_{C}$ is the probability measure associated to $\left(B^1, B^2\right)$ and $C = \left(C_t\right)_{t \geq 0} \in \mathcal{C}_B$ is the copula of $\left(B^1,B^2\right)$. In this section, we are interested in the distribution of the difference between $B^1$ and $B^2$, i.e. the function $x \mapsto \mathbb{P}_{C}\left(B_t^1 - B_t^2 \geq x\right)$ for $t > 0$ and in particular in the right tail of this distribution, i.e. when $x > 0$. Since the distributions of $B^1$ and $B^2$ are known, this function is entirely determined by the copula of $\left(B^1,B^2\right)$. Our goal is to find the range of values that can be achieved by this function at a given $x > 0$. Given $\eta > 0$ and $t \geq 0$, we define the function 
\begin{equation} \label{functionS}\begin{array}{ccccc}
S_{\eta,t} & : & \mathcal{C}_B & \to & \left[0,1\right] \\
 & & C & \mapsto & \mathbb{P}_{C}\left(B^1_t - B^2_t \geq \eta\right). \\
\end{array}
\end{equation}

\begin{remark} $\mathbb{P}_C$ is a probability measure that verifies $\mathbb{P}_C\left(B^1_t \leq x, B^2_t \leq y\right) = C_t\left(\Phi\left(\frac{x}{\sqrt{t}}\right),\Phi\left(\frac{y}{\sqrt{t}}\right)\right)$ for $x, y \in \mathbb{R}$. However, $C$ does not describe entirely $\mathbb{P}_{C}$. Indeed, $C$ describe the dependence between $B^1_t$ and $B^2_t$ at a given time $t$ but not between $B^1_s$ and $B^2_t$ with $s \neq t$ for instance.
\end{remark}

Our objective is to control the value of this function at a given time $t$ by controlling the dependence between the two Brownian motions. For this, we first study the range of this function $Ran\left(S_{\eta,t}\right)$. We show that the Reflection Brownian Copula defined in Section \ref{reflectionbrowniancopula} and its extensions allow us to control $S_{\eta,t}$ and to achieve all the values in $Ran\left(S_{\eta,t}\right)$. After showing the limitations of symmetric copulae for the control of $S_{\eta,t}$ in Section \ref{symmetry}, we give a result about $Ran\left(S_{\eta,t}\right)$ in a static case in Section \ref{spreadstatic}, i.e. in the case of two Gaussian random variables. Most of results of Section \ref{symmetry} and Section \ref{spreadstatic} are classic for the sum of random variables ; we adapt them to the difference case. Finally, we give the main result concerning the range of $S_{\eta,t}$ in Section \ref{spreaddynamic}.

\subsection{Impact of symmetry on $S_{\eta,t}$}
\label{symmetry}
\medskip

In this section, we show that modeling the dependence between two random variables with symmetric copulae limits the values that can be taken by the distribution of the difference between two random variables. It imposes some constraints on this distribution. Using asymmetric copulae is then necessary to control $S_{\eta, t}$. We also show that we can find asymmetric copulae suitable for Brownian motions.

\begin{definition} \label{copulasymmetric} A copula C is symmetric if $C\left(u,v\right) = C\left(v,u\right)$, $u, v \in \left[0,1\right]$. We denote by $\mathcal{C}_{s}$ the set of symmetric copulae.
\end{definition}

Note that $\mathcal{C}_G \subset \mathcal{C}_{s}$ with $\mathcal{C}_G$  the set of Gaussian copulae.

\bigskip
If $X$ and $Y$ are two random variables with continuous cumulative distribution functions, we denote by $C^{X,Y}$ the copula of $(X,Y)$. Sklar's Theorem \citep{sklar59} guarantees the existence and the unicity of $C^{X,Y}$. Proposition \ref{symmetry0} gives properties on the distribution the difference of two random variables if their copula is symmetric.

\medskip
\begin{proposition} \label{symmetry0}  \label{symmetryprop} Let $X$ and $Y$ be two real valued random variables defined on the same probability space $\left(\Omega,\; \mathcal{F},\; \mathbb{P}\right)$ with copula $C^{X,Y}$ and with continuous marginal distribution functions $F^X$ and $F^Y$. If $F^X = F^Y$ and $C^{X,Y} \in \mathcal{C}_{s}$ then $\mathbb{P}\left(X-Y \leq -x\right) = \mathbb{P}\left(X - Y \geq x\right)$.
\end{proposition}

\medskip
We can extend the definition of symmetry and asymmetry to Markov Diffusion Copulae: we denote by $\mathcal{C}_a^d = \{(C_t)_{t \geq 0} : \forall t \geq 0, C_t \in \mathcal{C}_s \}$ the set of symmetric Markov Diffusion Copulae and by $\mathcal{C}_s^d = \{(C_t)_{t \geq 0} : \forall t \geq 0, C_t \in \mathcal{C} \setminus \mathcal{C}_s \}$ the set of asymmetric Markov Diffusion Copulae.

\medskip
\begin{corollary} For $\eta > 0$ and $t > 0$, we have:
\[Ran\left(\restriction{S_{\eta,t}}{\mathcal{C}^d_s}\right) \subset \Bigl[0,\frac{1}{2}\Bigr]\]
with $\restriction{S_{\eta,t}}{\mathcal{C}^d_s}$ the restriction of $S_{\eta,t}$ to $\mathcal{C}^d_s$.
\end{corollary}

\begin{preuve} If we consider two Brownian motions $B^1$ and $B^2$ with dynamic copula $C \in \mathcal{C}^d_s$, we have according to Proposition \ref{symmetry0}: $\mathbb{P}\left(B^1_t - B^2_t \geq x\right) = \mathbb{P}\left(B^1_t - B^2_t \leq -x\right)$. However, $\mathbb{P}\left(B^1_t-B^2_t \geq x\right) + \mathbb{P}\left(B^1_t-B^2_t \leq -x\right) \leq 1$ if $x \geq 0$. Then we have the constraint $\mathbb{P}\left(B^1_t-B^2_t \geq x\right) \leq \frac{1}{2}$.
\end{preuve}

In particular, since the Gaussian copula is symmetric, it is not possible to obtain asymmetry in the distribution of $B^1_t -B^2_t$ at each time $t$ when the dependence between two Brownian motions is given by a correlation structure. Limiting the modeling of the dependence to the Gaussian copula or to symmetric copulae makes the distribution of their difference symmetric and limits the value of $S_{\eta,t}$.

\medskip
Modeling the dependence by an asymmetric copula is then necessary to have higher values than $\frac{1}{2}$ for $S_{\eta,t}$. We have
\[\mathcal{C}_{B} \cap \mathcal{C}_a^d \neq \emptyset.\]
Indeed, the Reflection Brownian Copula defined in Equation \eqref{copulareflectioneq} is in $\mathcal{C}_B$ and is asymmetric. The set of admissible copulae for Brownian motion is not reduced to the set of Gaussian copulae and furthermore it contains an asymmetric copula which is the Reflection Brownian Copula. Jaworski and Krzywda \citep{jaworski13} and Bosc \citep{bosc12} have proven the existence of symmetric suitable copulae for Brownian motions. However, they did not find asymmetric copulae  suitable for Brownian motions. We can also show that extensions of the Brownian Reflection Copula defined in Section \ref{extensions} are asymmetric. To our knowledge, these copulae are the only asymmetric copulae suitable for Brownian motions in the literature.

\begin{remark} Copulae constructed in Section \ref{extensions} can also be used as a method to construct asymmetric copulae, which is not always evident.  
\end{remark}

\subsection{The Gaussian Random Variables Case}
\label{spreadstatic}
\medskip

Let us consider two standard normal random variables $X$ and $Y$ defined on a common probability space ($\Omega$, $\mathcal{F}$, $\mathbb{P}_C$) where $\mathbb{P}_{C}$ is the probability measure associated to the copula $C$ of $\left(X, Y\right)$. Since the laws of the marginals of $X$ and $Y$ are fixed, the probability measure only depends on the copula of $\left(X,Y\right)$, which justifies the notation $\mathbb{P}_{C}$. In this section, we study the control of the distribution of the difference $\mathbb{P}_C\left(X - Y \geq \eta\right)$ for a given $\eta$. We need to adapt the definition of $S_{\eta,t}$ for the static case, i.e. when the copula are not dynamic. We define the function
\[\begin{array}{ccccc}
\tilde{S}_{\eta} & : & \mathcal{C} & \to & \left[0,1\right] \\
 & & C & \mapsto & \mathbb{P}_{C}\left(X - Y \geq \eta\right) \\
\end{array}
\]
for a given $\eta > 0$.

\begin{remark} $\mathbb{P}_C$ is defined by $\mathbb{P}_C\left(X \leq x, Y \leq y\right) = C\left(\Phi\left(x\right),\Phi\left(y\right)\right)$ for $x, y \in \mathbb{R}$.
\end{remark}

\medskip

In particular, we look for an upper bound of $\tilde{S}_{\eta}$. Lower bound is trivial and is achieved by the copula $M\left(u,v\right) = \min\left(u,v\right)$. Note that this copula is equivalent to having correlation 1 between the two random variables and corresponds to a case of comonotonicity. The problem is similar to the one consisting in finding bounds on the distribution of the sum. Makarov \citep{makarov82} finds bounds on the cumulative distribution function of the sum of two random variables at a given point given the marginals. R{\"u}schendorf \citep{ruschendorf82} proves this result using optimal transport theory. Frank et al. \citep{frank87} prove the same result using copulae and find a copula that achieves the bound. Furthermore, the results are extended to dimensions greater than 2 and to the cumulative distribution function of $L(X,Y)$ where $L$ is a non decreasing continuous function in $X$ and $Y$ with $X$ and $Y$ two random variables. Finding these bounds have several applications in finance and insurance such as finding bounds on value-at-risk \citep{embrechts03}.

\bigskip
In Proposition \ref{boundcopula}, we study the range of values taken by $\tilde{S}_{\eta}$. In particular, we look for an upper bound when the copula is taken among the set of Gaussian copulae then among all the copulae. We also find the range of $\tilde{S}_{\eta}$. In order to maximize $\tilde{S}_{\eta}\left(C\right)$ over all the copulae, we use the approach of Frank et al. \citep{frank87} with copulae. 

\medskip
\begin{proposition} \label{boundcopula} Let $\eta > 0$.

\medskip
Let 
\[
C^{r}(u,v) = \begin{cases}
M\left(u -1 + r , v\right) & \text{ if $(u,v) \in \left[ 1-r,1\right] \times  \left[ 0,r\right]$,}\\ 
W\left(u,v\right) & \text{if $(u,v) \in \left[0,1\right]^2 \setminus \left(\left[ 1-r,1\right] \times  \left[ 0,r\right]\right)$}
\end{cases}
\]
with $r =  2\Phi\left(\frac{-\eta}{2}\right)$.

\medskip
We have:
\begin{enumerate}
\item[(i)]  $Ran\left(\restriction{\tilde{S}_{\eta}}{\mathcal{C}_G}\right) = \left[0, \Phi\left(\frac{-\eta}{2}\right) \right]$ with $\restriction{\tilde{S}_{\eta}}{\mathcal{C}_G}$ the restriction of $\tilde{S}_{\eta}$ to $\mathcal{C}_G$,
\item[(ii)]  $\underset{C \in \mathcal{C}}{\sup \;} \tilde{S}_{\eta}\left(C\right) =  2\Phi\left(\frac{-\eta}{2}\right)$ and the supremum is achieved with $C^{r}$,
\item[(iii)]  $Ran\left(\tilde{S}_{\eta}\right) = \left[0,2\Phi\left(\frac{-\eta}{2}\right) \right]$.
\end{enumerate}
\end{proposition}
\medskip

\medskip
If we only consider the set of Gaussian copulae, $\tilde{S}_{\eta}$ can only achieve the values in $\left[0, \Phi\left(\frac{-\eta}{2}\right) \right]$. If we consider all the copulae, values in $\left[\Phi\left(\frac{-\eta}{2}\right), 2\Phi\left(\frac{-\eta}{2}\right) \right]$ can also be achieved. Indeed, we can use the family of copulae constructed in Proposition \ref{boundcopula} to achieve these values. It has a particular structure: it is divided in two parts according to the value of the first random variable. One state corresponds to a positive correlation and the upper bound is achieved in the comonotonic case. The other state corresponds to the countermonotonic case.

\medskip
The family of copulae constructed in Proposition \ref{boundcopula} are patchwork copulae \citep{durante13}. Given a copula $C$, a patchwork copula is constructed by changing the value of $C$ in a subrectangle of the unit square and replacing it with an other copula. In our case, we consider the countermonotonic copula and we change its values in the rectangle $\left[1-r,1\right] \times \left[0,r\right]$, replacing it by a Gaussian copula with parameter $\rho$. The copula achieving the bound corresponds to $\rho = 1$ and in this particular case, the copula is called a shuffle of $M$ copula \citep{nelsen06}. Figure \ref{copulapatchworkfig} shows illustration of the copulae family constructed in Proposition \ref{boundcopula} with a correlation of 1 and a correlation of $-0.95$.

\begin{figure}[h!]
    \centering
    \begin{subfigure}[b]{0.32\textwidth}
        \centering
        \includegraphics[width=\textwidth]{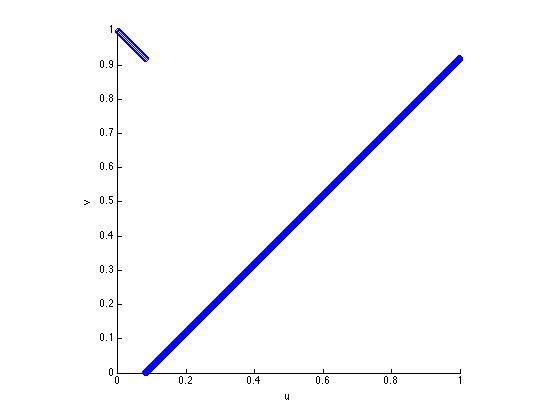}
        \caption{\label{copulapatchworkfig1} \it $\rho = 1$}
    \end{subfigure}
    \begin{subfigure}[b]{0.32\textwidth}
        \centering
        \includegraphics[width=\textwidth]{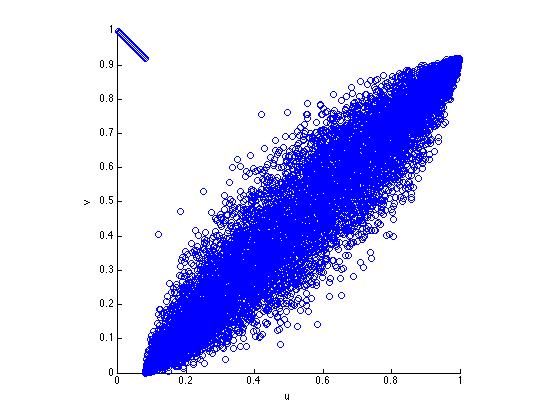}
        \caption{\it $\rho = 0.95$}
        \end{subfigure}
    \caption{\it \label{copulapatchworkfig} Patchwork copula $C^r(u,v)$ presenting two states depending on the value of u: the first copula is in the upper left part of the plan and is $W$ ; the second one is the Gaussian copula with correlation equal to $\rho$, with $\rho = 1$ that is the degenerated copula $M$ or $\rho = 0.95$. $r$ is equal to $2\Phi\left(\frac{-\eta}{2}\right)$ with $\eta = 0.2$.}
 \end{figure}

\bigskip
If we consider two Brownian motions $B^1$ and $B^2$, $B^1_t$ and $B^2_t$ at a given time $t$ are Gaussian random variables with variance $t$. Proposition \ref{boundcopula} can be applied with $\eta^{'} = \frac{\eta}{\sqrt{t}}$. Modeling the dependence of Brownian motions with a Gaussian copula then limits us in terms of values taken by $S_{\eta,t}$. In particular, it is not possible to have probabilities greater than $\frac{1}{2}$ which was already proven with the symmetry property of Gaussian copulae.

\bigskip
In this section, we showed the limits of the Gaussian copulae and that it was possible to achieve new values for $\tilde{S}_{\eta}$ or to put asymmetry in the distribution of the difference with the use of different types of copulae. However, the copulae were used to model the dependence between the two Gaussian variables, i.e. two Brownian motions at given time $t$. We do not know if the copulae are suitable to model the dependence between $B^1 = (B^1_t)_{t \geq 0}$ and $B^2 = (B^2_t)_{t \geq 0}$, that is in a dynamical framework. 

\subsection{The Brownian Motion Case}
\label{spreaddynamic}
\medskip
Proposition \ref{boundcopuladynamic} gives a time dynamical version of Proposition \ref{boundcopula}.

\begin{proposition} \label{boundcopuladynamic} Let $\eta > 0$ and $t > 0$. We have:
\begin{enumerate}
\item[(i)]  $Ran\left(\restriction{S_{\eta,t}}{\mathcal{C}^d_G}\right) = \left[0, \Phi\left(\frac{-\eta}{2\sqrt{t}}\right) \right]$ with $\restriction{S_{\eta,t}}{\mathcal{C}^d_G}$ the restriction of $S_{\eta,t}$ to $\mathcal{C}^d_G$,
\item[(ii)]  $\underset{C \in \mathcal{C}_B}{\sup \;} S_{\eta,t}\left(C\right) =  2\Phi\left(\frac{-\eta}{2\sqrt{t}}\right)$ and the supremum is achieved with $C^{ref,\frac{\eta}{2}}$ which is the Reflection Brownian Copula defined by Equation \eqref{copulareflectioneq},
\item[(iii)]  $Ran\left(S_{\eta,t}\right) = \left[0, 2\Phi\left(\frac{-\eta}{2\sqrt{t}}\right) \right]$.
\end{enumerate}
\end{proposition}

\medskip
We have found a copula  which maximizes $S_{\eta,t}$ at each time $t$ which is admissible for Brownian motions. This copula is also a solution to the problem $\underset{C \in \mathcal{C}}{\sup \;} \tilde{S}_{\frac{\eta}{\sqrt{t}}}\left(C\right)$ and gives an alternative solution of the supremum copula of Proposition \ref{boundcopula}. We also notice than  $Ran\left(S_{\eta,t}\right) = Ran\left(\tilde{S}_{\frac{\eta}{\sqrt{t}}}\right)$. The constraint to be in $\mathcal{C}_B$ does not change the solution of our problem, values that can be achieved are the same ; only copulae differ. Furthermore, copulae used to achieved the range of $S_{\eta,t}$ gives Markovian pairs of Brownian motions: the range of possible values for $\mathbb{P}\left(B^1_t - B^2_t \geq \eta\right)$ is the same for Markovian pairs and all pairs of Brownian motions. As in the copula of Proposition \ref{boundcopula}, the copula that achieves the supremum has two states: one of comonotonicity and one of countermonotonicity, depending here on the value of the $B^1_t - B^2_t$. Figure \ref{reflectioncopulafig} represents the Reflection Brownian Copula at time $t = 1$ with a reflection at $\frac{\eta}{2} = 0.1$. We can see that the structure is the same than the copula of Figure \ref{copulapatchworkfig1}. However, in Figure  \ref{copulapatchworkfig1} the two lines are in separated parts of the square and in Figure \ref{reflectioncopulafig}, there is a part of the square where they are both present. This is due to the fact that $\tilde{B_t}$ is not a deterministic function of $B_t$ but a deterministic function of $B_t$ and $\underset{s \leq t}{\text{sup }}B_s$.  

\begin{figure}[h!]
\centering
\includegraphics[width=0.4\textwidth]{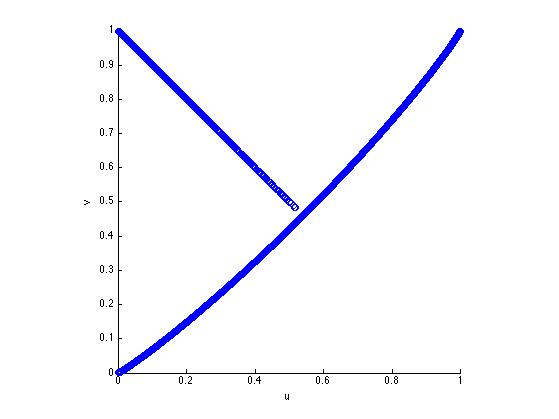}
\caption{\label{reflectioncopulafig} \it Reflection Brownian Copula $C_t^{ref,\frac{\eta}{2}}$ at time $t=1$ with $\eta = 0.2$.}
\end{figure}

Part (iii) of Proposition \ref{boundcopuladynamic} gives us a way to control  $S_{\eta,t}$. Furthermore, when the copula is the Reflection Brownian one, the probability for $B^1_t - B^2_t$ to have strictly higher value than $\eta$ is equal to 0 and there is a discontinuity at $\eta$ ; copulae of part (iii) allow us to solve this issue. The copulae become suboptimal but still achieves higher values than in the Gaussian copula case.

\bigskip
Result of Proposition \ref{boundcopuladynamic} (ii) can be interpreted with coupling. Let $X$ be a stochastic process. Let $X^{a}$ and $X^{b}$ be processes with the dynamic of $X$ such that $X^a_0 = a$ and $X^b_0 = b$. A coupling is said successful if $T = \inf \{ t \geq 0: X_t^a = X_t^b\} < \infty$ almost surely. $T$ is called the coupling time. In our situation, the two Brownian motions start at 0 and are coupled when $B_t$ = $\tilde{B}^{\frac{\eta}{2}}_t  + \eta$ which is equivalent to consider one Brownian starting at 0 and the other starting at $\eta$. We have the coupling inequality:

\begin{equation}
\label{couplinginequality}
\| \mathbb{Q}_a\left(t\right) - \mathbb{Q}_b\left(t\right) \| \leq 2 \mathbb{P}\left(T > t\right)
\end{equation}
with $\| \|$ the total variation norm and $\mathbb{Q}_a(t)$ the distribution of  $X_t^a$ (same for $\mathbb{Q}_b(t)$ and $X^t_b$). In case of equality for \eqref{couplinginequality}, the coupling is said to be optimal \citep{hsu13}. The coupling by reflection \citep{lindvall86}, consisting of taking the reflection of the Brownian motion according to the hyperplane $x = \frac{a + b}{2}$, is optimal for Brownian motion. Hsu and Sturm \citep{hsu13} prove that in the case of Brownian motions, it is the only optimal Markovian coupling (definition \ref{markoviancoupling}).

\begin{definition} \label{markoviancoupling} \citep{hsu13} Let $X = (X_1,X_2)$ be a coupling of Brownian motions. Let $\mathcal{F}^X $ be the filtration generated by X. We say that X is a Markovian coupling if for each $s \geq 0$, conditional on $\mathcal{F}_s^X$, the shifted process $\{(X_1(t+s),X_2(t+s)), t \geq 0\}$ is still a coupling of Brownian motions (now starting from $(X_1(s), X_2(s))$).
\end{definition}

\medskip
In the optimal case, $\mathbb{P}\left(T > t\right)$ is minimal. The coupling by reflection can then be interpreted as the fastest way for the two processes to be equal. In our case, it is the fastest way for the $B^1 - B^2$ to be greater than $\eta$.

\bigskip
We found an admissible copula for Brownian motions which has the property to be asymmetric and to achieve upper bound for $S_{\eta,t}$. We have also constructed new families of asymmetric copulae allowing us to control the value of $S_{\eta,t}$.  

\section{Proofs}
\label{proofs}

\subsection{Preliminary results}

We start with well known results that will be useful for the proofs of propositions.

\medskip
\begin{lemma} \label{lawsup} \label{lawinf} Let $B = \left(B_t\right)_{t \geq 0}$ be a standard Brownian motion on a filtered probability space $\left(\Omega, \mathcal{F}, \left(\mathcal{F}_t\right)_{t\geq0}, \mathbb{P}\right)$. We have, for $y \geq 0$, 
\[\mathbb{P}\Bigl(B_t \leq x,  \underset{s \leq t}{\sup \;} B_s \leq y\Bigr) = \left\lbrace
\begin{array}{lll}
\Phi\Bigl(\frac{x}{\sqrt{t}}\Bigr) - \Phi\Bigl(\frac{x-2y}{\sqrt{t}}\Bigr) & \mbox{if} & x < y  \\
2\Phi\Bigl(\frac{y}{\sqrt{t}}\Bigr) - 1  & \mbox{if} & x \geq y 
\end{array}\right..  \]

\end{lemma}

\begin{preuve} The reader is referred to \cite[Theorem\ 3.1.1.2, p.\ 137]{jeanblanc09}.
\end{preuve}

\subsection{Proof of Proposition \ref{copulareflection}}
\label{reflectionbrowniancopulaproofs}

We have:
\begin{equation}
\label{propositioncopulareflection1}
\mathbb{P}\Bigl(B_t \leq x, \tilde{B}^{h}_t \leq y\Bigr) = \mathbb{P}\Bigl(B_t \leq x, \tilde{B}^{h}_t \leq y, \underset{s \leq t}{\sup \;} B_s \leq h\Bigr) + \mathbb{P}\Bigl(B_t \leq x, \tilde{B}^{h}_t \leq y, \underset{s \leq t}{\sup \;} B_s \geq h\Bigr).
\end{equation}
We compute the first term of Equation \eqref{propositioncopulareflection1}:
\begin{align*}
\mathbb{P}\Bigl(B_t \leq x, \tilde{B}^{h}_t \leq y, \underset{s \leq t}{\sup \;} B_s \leq h\Bigr) &= \mathbb{P}\Bigl(B_t \leq x, -B_t \leq y, \underset{s \leq t}{\sup \;} B_s \leq  h\Bigr) \\
& = \mathbb{P}\Bigl( -y \leq B_t \leq x, \underset{s \leq t}{\sup \;} B_s \leq  h\Bigr) \\
& = \Bigl(\mathbb{P}\Bigl(B_t \leq x, \underset{s \leq t}{\sup \;} B_s \leq  h\Bigr) - \mathbb{P}\Bigl(B_t \leq -y, \underset{s \leq t}{\sup \;} B_s \leq  h\Bigr)\Bigr){\bf1}_{x+y > 0}\\
&= \left\lbrace
\begin{array}{ccc}
\Phi\Bigl(\frac{x}{\sqrt{t}}\Bigr) - \Phi\Bigl(\frac{x-2h}{\sqrt{t}}\Bigr) + \Phi\Bigl(\frac{y}{\sqrt{t}}\Bigr) - \Phi\Bigl(\frac{y+2h}{\sqrt{t}}\Bigr) & \mbox{if} &  x \leq h, \\
& & y \geq -h,\\
& & x + y > 0 \\
2\Phi\Bigl(\frac{h}{\sqrt{t}}\Bigr) - 1 + \Phi\Bigl(\frac{y}{\sqrt{t}}\Bigr) - \Phi\Bigl(\frac{y+2h}{\sqrt{t}}\Bigr)  & \mbox{if} & x > h, \\
& & y \geq -h \\
0 & \mbox{} & \mbox{otherwise}
\end{array}\right.
\end{align*}
by application of Lemma \ref{lawsup}. In the same way, we compute the second term of Equation \eqref{propositioncopulareflection1}: 
\begin{align*}
\mathbb{P}\Bigl(B_t \leq x, \tilde{B}^{h}_t \leq y, \underset{s \leq t}{\sup \;} B_s \geq h\Bigr) &= \mathbb{P}\Bigl(B_t \leq x, B_t \leq y + 2h, \underset{s \leq t}{\sup \;} B_s \geq h\Bigr) \\
&= \mathbb{P}\Bigl(B_t \leq \min\left(x, y + 2h\right), \underset{s \leq t}{\sup \;} B_s \geq h\Bigr)  \\
& = \left\lbrace
\begin{array}{ccc}
\Phi\Bigl(\frac{\min\left(x,y+2h\right)-2h}{\sqrt{t}}\Bigr) & \mbox{if} &  \min\left(x,y+2h\right) < h \\
- 2\Phi\Bigl(\frac{h}{\sqrt{t}}\Bigr) + 1 + \Phi\Bigl(\frac{\min\left(x,y+2h\right)}{\sqrt{t}}\Bigr) & \mbox{if} & \min\left(x,y+2h\right) \geq h
\end{array}\right..
\end{align*}
Combining the last two equations, we obtain 
\begin{align*}
 \mathbb{P}\left(B_t \leq x, \tilde{B}^{h}_t \leq y\right) &= \left\lbrace
\begin{array}{c}
\Phi\left(\frac{\min\left(x,y+2h\right)-2h}{\sqrt{t}}\right) \text{ if }  x + y \leq  0 \text{ or } \left(y \leq - h, x + y > 0\right) \\
 \Phi\left(\frac{\min\left(x,y+2h\right)}{\sqrt{t}}\right) - \Phi\left(\frac{y+2h}{\sqrt{t}}\right) + \Phi\left(\frac{y}{\sqrt{t}}\right)   \text{ if }  y > - h, \; x + y > 0\\
\end{array}\right. \\
\numberthis \label{refcopulainter}& = \left\lbrace
\begin{array}{ccc}
\Phi\Bigl(\frac{y}{\sqrt{t}}\Bigr) \text{ if } x - y \geq 2h \\
\Phi\Bigl(\frac{x-2h}{\sqrt{t}}\Bigr) \text{ if } x - y < 2h, x + y \leq 0 \\
\Phi\Bigl(\frac{x}{\sqrt{t}}\Bigr) - \Phi\Bigl(\frac{y+2h}{\sqrt{t}}\Bigr) + \Phi\Bigl(\frac{y}{\sqrt{t}}\Bigr) \text{ if }  x - y < 2h, x + y > 0 
\end{array}\right. \\
& =  \left\lbrace
\begin{array}{ccc}
\Phi\Bigl(\frac{y}{\sqrt{t}}\Bigr) \text{ if } x - y \geq 2h \\
W\Bigl(\Phi\left(\frac{y}{\sqrt{t}}\Bigr) + \Phi\Bigl(\frac{x}{\sqrt{t}}\right)\Bigr) + M\Bigl(\Phi\Bigl(\frac{x-2h}{\sqrt{t}}\Bigr), \Phi\Bigl(\frac{- y - 2h}{\sqrt{t}}\Bigr)\Bigr) \text{ if } x - y < 2h
\end{array}\right..
\end{align*}
We conclude using $C^{ref,h}_t\left(u,v\right) = \mathbb{P}\Bigl(B_t \leq \sqrt{t}\Phi^{-1}\left(u\right), \tilde{B}^{h}_t \leq \sqrt{t}\Phi^{-1}\left(v\right)\Bigr)$.

\subsection{Proof of Proposition \ref{nondegenerated}}

Recall that $\Phi_{\rho}$ denotes the bivariate cumulative distribution function of two standard normal variables correlated with $\rho \in \left[-1,1\right]$. We start with a technical lemma. 

\begin{lemma} \label{intphi} Let a, b and x $\in \mathbb{R}$. We have:
\begin{enumerate}
\item[(i)] \[\int_{-\infty}^{x} \Phi\left(a u + b\right) \frac{e^{\frac{-u^2}{2}}}{\sqrt{2\pi}} du = \Phi_{\frac{-a}{\sqrt{a^2+1}}}\Bigl(\frac{b}{\sqrt{a^2+1}}, x\Bigr).\]
\item[(ii)] \[\Phi_{\sqrt{1-\rho^2}}\left(x,y\right) = \Phi\left(y\right)\Phi\Bigl(\frac{x-\sqrt{1-\rho^2}y}{\rho}\Bigr) + \Phi\left(x\right) - \Phi_{\rho}\Bigl(x,\frac{x-\sqrt{1-\rho^2}y}{\rho}\Bigr),\; x, y \in \mathbb{R},\; \rho > 0\]
\item[(iii)]  \[\Phi_{\rho}\left(x,y\right) = \Phi\left(y\right) - \Phi_{-\rho}\left(-x,y\right),\; x, y \in \mathbb{R}\] 
\end{enumerate}
\end{lemma}

\begin{preuve}

\noindent (ii) Let $a < 0$, $b, z \in \mathbb{R}$. We have:
\begin{align*}
\Phi_{\frac{-a}{\sqrt{a^2+1}}}\Bigl(\frac{b}{\sqrt{a^2+1}}, z\Bigr) &= \int_{-\infty}^{z} \Phi\left(a u + b\right) \frac{e^{\frac{-u^2}{2}}}{\sqrt{2\pi}} du  \\
&=\Phi\left(az+b\right)\Phi\left(z\right) - a\int_{-\infty}^{z} \frac{e^{\frac{-\left(au+b\right)^2}{2}}}{\sqrt{2\pi}}  \Phi\left(u\right) du\\
&=\Phi\left(az+b\right)\Phi\left(z\right)  + \int_{az+b}^{+\infty} \frac{e^{\frac{-u^2}{2}}}{\sqrt{2\pi}}  \Phi\Bigl(\frac{u-b}{a}\Bigr) du\\
&= \Phi\left(az+b\right)\Phi\left(z\right) + \Phi\Bigl(\frac{b}{\sqrt{1+a^2}}\Bigr) - \Phi_{\frac{1}{\sqrt{a^2+1}}}\Bigl(\frac{b}{\sqrt{a^2+1}}, ax + b\Bigr). 
\end{align*}
We conclude by taking $a = \frac{-\rho}{\sqrt{1-\rho^2}}$, $b = x\sqrt{1+a^2}$, $z = \frac{y-b}{a}$.

\smallskip
\noindent (iii) 
Let $\left(X,Y\right)$ be a Gaussian vector with $X$ and $Y$ having correlation $\rho$. We have:
\[\mathbb{P}\left(X \leq x, Y \leq y\right) = \mathbb{P}\left( - X \geq - x, Y \leq y\right) =\mathbb{P}\left(Y \leq y\right) - \mathbb{P}\left( - X \leq - x, Y \leq y\right).\]
\end{preuve}

We can now prove Proposition \ref{nondegenerated}.
\medskip
Let $X = B$ and $Y = \rho \tilde{B}^{h} + \sqrt{1-\rho^2} Z$ where $B$ and $Z$ are two independent Brownian motions. $X$ and $Y$ are Brownian motions and we have 
\begin{align*}
\mathbb{P}\left(X_t \leq x, Y_t \leq y\right) &= \mathbb{P}\Bigl(B_t \leq x, \tilde{B}^{h}_t \leq \frac{y - \sqrt{1-\rho^2} Z_t}{\rho}\Bigr)\\
&= \mathbb{E}\Bigl[\mathbb{P}\Bigl(B_t \leq x, \tilde{B}^{h}_t \leq \frac{y - \sqrt{1-\rho^2} Z_t}{\rho} \mid Z_t\Bigr)\Bigr]. 
\end{align*}
Since $B$ is independent from $Z$, using Equation \eqref{refcopulainter}, we find that $\mathbb{P}\left(X_t \leq x, Y_t \leq y\right)$ is the sum of the three following terms:
\begin{enumerate}
\item[(i)]  \[\mathbb{E}\Bigl[\Phi\Bigl(\frac{y - \sqrt{1-\rho^2} Z_t}{\rho\sqrt{t}}\Bigr){\bf1}_{Z_t \geq \frac{\rho\left(2h -x\right) + y}{\sqrt{1-\rho^2}}}\Bigr],\]
\item[(ii)]  \[\mathbb{E}\Bigl[\Phi\Bigl(\frac{x-2h}{\sqrt{t}}\Bigr){\bf1}_{Z_t \leq \frac{\rho\left(2h -x\right) + y}{\sqrt{1-\rho^2}}}{\bf1}_{Z_t \geq \frac{\rho x + y}{\sqrt{1-\rho^2}}} \Bigr],\]
\item[(iii)]  \[\mathbb{E}\Bigl[ \Bigl( \Phi\Bigl(\frac{x}{\sqrt{t}}\Bigr) - \Phi\Bigl(\frac{y + 2h\rho - \sqrt{1-\rho^2} Z_t}{\rho\sqrt{t}}\Bigr) + \Phi\Bigl(\frac{y - \sqrt{1-\rho^2} Z_t}{\rho\sqrt{t}}\Bigr) \Bigr){\bf1}_{Z_t \leq \frac{\rho\left(2h -x\right) + y}{\sqrt{1-\rho^2}}} {\bf1}_{Z_t \leq \frac{\rho x + y}{ \sqrt{1-\rho^2} } } \Bigr].\] 
\end{enumerate}

The first term (i) is equal to:
\[\mathbb{E}\Bigl[\Phi\Bigl(\frac{y - \sqrt{1-\rho^2} Z_t}{\rho\sqrt{t}}\Bigr)\Bigr] - \mathbb{E}\Bigl[\Phi\Bigl(\frac{y - \sqrt{1-\rho^2} Z_t}{\rho\sqrt{t}}\Bigr){\bf1}_{Z_t \leq \frac{\rho\left(2h -x\right) + y}{\sqrt{1-\rho^2}}}\Bigr].\]

Furthermore, using Lemma \ref{intphi}, we have 
\begin{align*}
\mathbb{E}\Bigl[\Phi\Bigl(\frac{y - \sqrt{1-\rho^2} Z_t}{\rho\sqrt{t}}\Bigr){\bf1}_{Z_t \leq \frac{\rho\left(2h -x\right) + y}{\sqrt{1-\rho^2}}}\Bigr] &= \int_{-\infty}^{\frac{\rho\left(2h -x\right) + y}{\sqrt{\left(1-\rho^2\right)t}}} \Phi\Bigl(\frac{y - \sqrt{1-\rho^2}\sqrt{t}u}{\rho\sqrt{t}}\Bigr) \frac{e^{\frac{-u^2}{2}}}{\sqrt{2\pi}} du\\
&= \Phi_{\sqrt{1-\rho^2}}\Bigl(\frac{y}{\sqrt{t}}, \frac{\rho(2h -x) + y}{\sqrt{\left(1-\rho^2\right)t}}\Bigr)
\end{align*}
and 
\[\mathbb{E}\Bigl[\Phi\Bigl(\frac{y - \sqrt{1-\rho^2} Z_t}{\rho\sqrt{t}}\Bigr)\Bigr] =  \Phi\Bigl(\frac{y}{\sqrt{t}}\Bigr).\]
We compute terms (ii) and (iii) using the same method and we find that $\mathbb{P}\Bigl(X_t \leq x, Y_t \leq y\Bigr)$ is equal to 
\begin{align*}
 &\Phi\Bigl(\frac{y}{\sqrt{t}}\Bigr) - \Phi_{\sqrt{1-\rho^2}}\Bigl(\frac{y}{\sqrt{t}}, \frac{\rho\left(2h -x\right) + y}{\sqrt{\left(1-\rho^2\right)t}}\Bigr) + \Phi_{\sqrt{1-\rho^2}}\Bigl(\frac{y}{\sqrt{t}}, \frac{\rho\min\left(2h -x,x\right) + y}{\sqrt{\left(1-\rho^2\right)t}}\Bigr)\\
& - \Phi_{\sqrt{1-\rho^2}}\Bigl(\frac{y+2h\rho}{\sqrt{t}}, \frac{\rho\min\left(2h -x,x\right) + y}{\sqrt{\left(1-\rho^2\right)t}}\Bigr) + \Phi\Bigl(\frac{x}{\sqrt{t}}\Bigr) \Phi\Bigl(\frac{\rho\min\left(2h -x,x\right) + y}{\sqrt{\left(1-\rho^2\right)t}}\Bigr) \\
&+ \Phi\Bigl(\frac{x-2h}{\sqrt{t}}\Bigr)\Bigl(\Phi\Bigl(\frac{\rho\left(\max\left(2h -x,x\right)\right) + y}{\sqrt{\left(1-\rho^2\right)t}}\Bigr) - \Phi\Bigl(\frac{\rho x + y}{\sqrt{\left(1-\rho^2\right)t}}\Bigr)\Bigr).\end{align*}

After some algebra, we find using Lemma \ref{intphi}:
\[\mathbb{P}\Bigl(X_t \leq x, Y_t \leq y\Bigr) = \left\lbrace
\begin{array}{c}
\Phi_{\rho}\Bigl(\frac{y+2\rho h}{\sqrt{t}}, \frac{x}{\sqrt{t}}\Bigr) + \Phi\Bigl(\frac{y}{\sqrt{t}}\Bigr) - \Phi\Bigl(\frac{y+2\rho h}{\sqrt{t}}\Bigr)  \text{ if }  x  \geq h \\
 \hspace{-0.5em}\Phi_{-\rho}\Bigl(\frac{y}{\sqrt{t}}, \frac{x}{\sqrt{t}}\Bigr) + \Phi_{\rho}\Bigl(\frac{-y-2\rho h}{\sqrt{t}}, \frac{x-2h}{\sqrt{t}}\Bigr) +  \Phi_{\rho}\Bigl(\frac{y}{\sqrt{t}}, \frac{x-2h}{\sqrt{t}}\Bigr) -\Phi\Bigl(\frac{x-2h}{\sqrt{t}}\Bigr) \text{ if }  x < h
\end{array}\right.
\]
and the copula is equal to $\mathbb{P}\left(X_t \leq \sqrt{t}\Phi^{-1}\left(u\right), Y_t \leq \sqrt{t}\Phi^{-1}\left(v\right)\right)$.

\subsection{Proof of Proposition \ref{randombarrier}}

Let $f^{\xi}$ be the density of $\xi$. Let $B$ be a Brownian motion independent from $\xi$. We enlarge the filtration of $B$ to take into account $\xi$. We consider the reflection of the Brownian motion $\tilde{B}^{\xi}$. We have:
\begin{equation} \label{proofrandombarrier1}\mathbb{P}\left(B_t \leq \sqrt{t}\Phi^{-1}\left(u\right), \tilde{B}^{\xi}_t \leq \sqrt{t}\Phi^{-1}\left(v\right)\right) = \mathbb{E}\left[\mathbb{P}\left(B_t \leq \sqrt{t}\Phi^{-1}\left(u\right), \tilde{B}^{\xi}_t \leq  \sqrt{t}\Phi^{-1}\left(v\right) \mid \xi\right)\right].
\end{equation}
Since $B$ is independent from $\xi$, we have according to Proposition \ref{copulareflection}:
\begin{align*}\mathbb{P}\left(B_t \leq \sqrt{t}\Phi^{-1}\left(u\right), \tilde{B}^{\xi}_t \leq  \sqrt{t}\Phi^{-1}\left(v\right) \mid \xi\right) =& v{\bf1}_{\Phi^{-1}\left(u\right) - \Phi^{-1}\left(v\right) \geq \frac{2\xi}{\sqrt{t}}} + W\left(u,v\right){\bf1}_{\Phi^{-1}\left(u\right) - \Phi^{-1}\left(v\right) < \frac{2\xi}{\sqrt{t}}} \\
& + \Phi\Bigl(\Phi^{-1}\left(M\left(u,1-v\right)\right) - \frac{2\xi}{\sqrt{t}}\Bigr){\bf1}_{\Phi^{-1}\left(u\right) - \Phi^{-1}\left(v\right) < \frac{2\xi}{\sqrt{t}}} 
\end{align*}
Thus, the right hand side of Equation \eqref{proofrandombarrier1} is the sum of the three following terms:
\begin{enumerate}
\item[(i)]  
\[\mathbb{E}\Bigl[v{\bf1}_{\Phi^{-1}\left(u\right) - \Phi^{-1}\left(v\right) \geq \frac{2\xi}{\sqrt{t}}}\Bigr] = vF^{\xi}\Bigl(\sqrt{t}\frac{\Phi^{-1}\left(u\right) - \Phi^{-1}\left(v\right)}{2}\Bigr),\]
\item[(ii)]  
\[\mathbb{E}\Bigl[W\left(u,v\right){\bf1}_{\Phi^{-1}\left(u\right) - \Phi^{-1}\left(v\right) < \frac{2\xi}{\sqrt{t}}}\Bigr] = W\left(u,v\right)\overline{F}^{\xi}\Bigl(\sqrt{t}\frac{\Phi^{-1}\left(u\right) - \Phi^{-1}\left(v\right)}{2}\Bigr),\]
\item[(iii)] 
\[\mathbb{E}\Bigl[\Phi\Bigl(\Phi^{-1}\left(M\left(u,1-v\right)\right) - \frac{2\xi}{\sqrt{t}}\Bigr){\bf1}_{\Phi^{-1}\left(u\right) - \Phi^{-1}\left(v\right) < \frac{2\xi}{\sqrt{t}}}\Bigr],\] that we denote  by $I$.
\end{enumerate}
We have:
\begin{align*}
I &= \int_{\sqrt{t}\frac{\Phi^{-1}\left(u\right) - \Phi^{-1}\left(v\right)}{2}}^{+\infty} \Phi\Bigl(\Phi^{-1}\left(M\left(u,1-v\right)\right) - \frac{2h}{\sqrt{t}}\Bigr)f^{\xi}\left(h\right)dh \\
&= M\left(1-u,v\right) \overline{F}^{\xi}\Bigl(\sqrt{t}\frac{\Phi^{-1}\left(u\right) - \Phi^{-1}\left(v\right)}{2}\Bigr) \\
&-\frac{2}{\sqrt{t}}\int_{\sqrt{t}\frac{\Phi^{-1}\left(u\right) - \Phi^{-1}\left(v\right)}{2}}^{+\infty} \Phi^{'}\Bigl(\Phi^{-1}\left(M\left(u,1-v\right)\right) - \frac{2h}{\sqrt{t}}\Bigr)\overline{F}^{\xi}\left(h\right)dh.
\end{align*}
Adding the three terms of Equation \eqref{proofrandombarrier1}, since $M\left(1-u,v\right) + W\left(u,v\right) = v$, we obtain:
\begin{align*}C^{\xi}_t\left(u,v\right)  &= v - \frac{2}{\sqrt{t}}\int_{\sqrt{t}\frac{\Phi^{-1}\left(u\right) - \Phi^{-1}\left(v\right)}{2} }^{+\infty} \Phi^{'}\Bigl(\Phi^{-1}\left(M\left(u,1-v\right)\right) - \frac{2h}{\sqrt{t}}\Bigr)\overline{F}^{\xi}\left(h\right)dh\\
&=  v - \int_{-\infty}^{\Phi^{-1}\left(M\left(1-u, v\right)\right)} \Phi^{'}\left(h\right)\overline{F}^{\xi}\Bigl(\frac{\sqrt{t}}{2}\left(\Phi^{-1}\left(M\left(u,1-v\right)\right) - h\right)\Bigr)dh
\end{align*}
with $\Phi^{'}\left(x\right) = \frac{e^{\frac{-x^2}{2}}}{\sqrt{2\pi}}$, which achieves the proof.

\subsection{Proof of Proposition \ref{symmetry0}}

\cite[Proposition\ 2.1]{cherubini11} states that
\begin{equation}
\label{sumintegral}
\mathbb{P}\left(X +Y \leq x\right) = \int_0^1 \partial_u C\left( u, F^Y\left(x - \left(F^X\right)^{-1}\left(u\right) \right) \right) du, x \in \mathbb{R}.
\end{equation}
The existence of $\partial_u C\left( u, F^Y\left(x - \left(F^X\right)^{-1}\left(u\right)\right)\right)$ for $u \in \left[0,1\right]$ is assured by \cite[Lemma\ 2.1]{cherubini11}.

We also have 
\begin{equation}
\label{eq1proofsymmetry}
F^{-Y}\left(y\right) = 1 - F^Y\left(-y\right), \; y \in \mathbb{R} \text{ and,}
\end{equation}
\begin{equation}
\label{eq2proofsymmetry}
C^{X,-Y}\left(u,v\right) = u - C\left(u, 1-v\right), \; \left(u,v\right) \in \left[0,1\right],
\end{equation}
with $C^{X,-Y}$ the copula of $(X,-Y)$.

Equation \eqref{sumintegral} is also valid for $(X,-Y)$. Using Equation \eqref{eq1proofsymmetry} and Equation \eqref{eq2proofsymmetry}, we have 
\begin{equation}
\label{diffcdf}
\mathbb{P}\left(X - Y \leq x\right) = \int_0^1 \left[1 - \partial_u C\left(u,F^Y\left(\left(F^X\right)^{-1}\left(u\right)-x\right)\right)\right]du, x \in \mathbb{R}
\end{equation}
and 
\begin{equation}
\label{diffintegral}
\mathbb{P}\left(X - Y > x\right) = \int_0^1 \partial_u C\left( u,F^Y\left(\left(F^X\right)^{-1}\left(u\right)-x\right) \right) du, x \in \mathbb{R}.
\end{equation}

\bigskip
Let us suppose that $C^{Y,X} \in \mathcal{C}_s$ and that $X$ and $Y$ have the same continous marginal distribution function $F$. Let $C^{Y,X}$ be the copula of $(Y,X)$. We have $C^{Y,X}\left(u,v\right) = C^{X,Y}\left(v,u\right)$. However, $C^{X,Y}\left(v,u\right) = C^{X,Y}\left(u,v\right)$ then $C^{Y,X}\left(u,v\right) = C^{X,Y}\left(u,v\right)$ and 
\begin{align*}
\mathbb{P}\left(X - Y \geq x\right) &= \mathbb{P}\left(Y-X \leq -x\right)\\
&= \int_0^1 \left[1 - \partial_u C^{Y,X}\left(u,F\left(\left(F\right)^{-1}\left(u\right)+x\right)\right)\right]du \\
&= \int_0^1 \left[1 - \partial_u C^{X,Y}\left(u,F\left(\left(F\right)^{-1}\left(u\right)+x\right)\right)\right]du\\
&=\mathbb{P}\left(X-Y \leq -x\right)
\end{align*}
using Equation \eqref{diffcdf}.

\subsection{Proof of Proposition \ref{boundcopula}}

\noindent (i) Let $\rho \in \left[-1, 1\right]$. We have $\tilde{S}_{\eta}\left(C_{G,\rho}\right)= \Phi\Bigl(\frac{-\eta}{\sqrt{2\left(1-\rho\right)}}\Bigr)$. This function is decreasing in $\rho$ and then the extremum are achieved for $\rho = 1$ and $\rho = -1$ and are equal to 0 and $\Phi\left(\frac{-\eta}{2}\right)$.

\smallskip
\noindent (ii) This is a direct application of the results of \citep{frank87} where superior and inferior bounds on $\mathbb{P}\left(X + Y < \eta\right)$ are found and where $X$ and $Y$ are two random variables with known marginals. 
As
\begin{enumerate}
\item[1)]  $\underset{C \in \mathcal{C}}{\sup \;} \mathbb{P}_{C}\left(X - Y \geq \eta\right) = 1 -  \underset{C \in \mathcal{C}}{\inf} \mathbb{P}_{C}\left(X - Y < \eta\right),$
\item[2)]  $-Y$ and $Y$ have the same law,
\end{enumerate}
the bound is equal to 
\[1 -  \underset{C \in \mathcal{C}}{\inf} \mathbb{P}_{C}\left(X + Y < \eta\right).\]
The copula achieving the bound is defined by the transformation 
\[C^{X,Y}\left(u,v\right) = u - C^{X,-Y}\left(u,1-v\right).\]

\smallskip
\noindent (iii) We want to prove that for all $x$ in $\left[0, 2\Phi\left(\frac{-\eta}{2}\right) \right]$, there exists $C$ in $\mathcal{C}$ such that $\tilde{S}_{\eta}\left(C\right) = x$.

\medskip
If $x \in \left[0, \Phi\left(\frac{-\eta}{2}\right) \right]$, we use a Gaussian copula with $\rho =  1 - \frac{1}{2^2}\left(\frac{\eta}{\Phi^{-1}\left(x\right)}\right)^2$.

\medskip
Let us suppose that $x \in \left[\Phi\left(\frac{-\eta}{2}\right), 2\Phi\left(\frac{-\eta}{2}\right) \right]$. We use the copula $C^{r}$  to construct a new class of copulae. As for $C^r$, we separate the square $\left[0, 1 \right]^2$ in two parts and to define a copula in each part of the square. We use the concept of patchwork copula defined by Durante et al. \citep{durante13}. Let $H =  \left[1-r,1\right] \times  \left[ 0,r\right]$, $H^c = \left[0, 1 \right]^2 \setminus H$ and $\rho \in \left(-1, 1\right)$. Let $C^p_{\rho}\left(u,v\right)$ the patchwork copula defined by $C_\rho$ in $H$ and $W$ in $H^c$: 
\begin{align*}
C^p_{\rho}\left(u,v\right) &= \mu_W\left(\left(\left[0,u\right] \times \left[0,v\right]\right) \cap H^c\right) + r C_{G,\rho}\Bigl(\frac{1}{r}\max\left(u+r-1,0\right),\frac{1}{r}\min \left(v,r\right)\Bigr)\\
&= \left(W\left(u,v\right) - W\left(u,r\right)\right){\bf1}_{v \geq r} + r C_{G,\rho}\Bigl(\frac{1}{r}\max \left(u+r-1,0\right),\min \bigl(\frac{v}{r},1\bigr)\Bigr)
\end{align*}
where $\mu_W$ is the measure induced by the copula $W$.

If we consider two standard normal random variables with copula $C^p_{\rho}$, the survival function of their difference at point $x$ is equal, according to Equation \eqref{diffintegral}, to
\begin{align*}
\int_0^1 \partial_u C^p_{\rho}\left( u, \Phi \left(\Phi^{-1}\left(u\right) - x\right) \right) du &= \int_0^1 \Big({\bf1}_{u \geq \Phi\left(\frac{x}{2} \right)} - {\bf1}_{u \geq 1 - r}\Big){\bf1}_{u \geq \Phi\left(\Phi^{-1}\left(r\right) + x\right)}du \\
&+  \int_{1-r}^1 \Phi\Bigl(\frac{\Phi^{-1}\bigl(\min \bigl(\frac{\Phi (\Phi^{-1}\left(u\right) - x)}{r},1\bigr)\bigr)-\rho \Phi^{-1}\left(\frac{u+r-1}{r}\right)}{\sqrt{1-\rho^2}}\Bigr)du\\ 
&= \left(1-r-\Phi\left(\frac{x}{2}\right)\right){\bf 1}_{x \leq 2\Phi^{-1}\left(1-r\right)} \\
&+  \int_{1-r}^1 \Phi\Bigl(\frac{\Phi^{-1}\bigl(\min \bigl(\frac{\Phi (\Phi^{-1}\left(u\right) - x)}{r},1\bigr)\bigr)-\rho \Phi^{-1}\left(\frac{u+r-1}{r}\right)}{\sqrt{1-\rho^2}}\Bigr)du\\ 
\end{align*}
which is continuous at $x = \eta$. Thus, $\tilde{S}_{\eta}\left(C^p_{\rho}\right)$ is equal to the survival function of their difference at point $\eta$, which is:
\[
\tilde{S}_{\eta}\left(C^p_{\rho}\right) = \int_{1-r}^1 \Phi\Bigl(\frac{\Phi^{-1}\bigl(\min \bigl(\frac{\Phi (\Phi^{-1}\left(u\right) - \eta)}{r},1\bigr)\bigr)-\rho \Phi^{-1}\left(\frac{u+r-1}{r}\right)}{\sqrt{1-\rho^2}}\Bigr)du.\\ 
\]
Using the previous equation and dominated convergence theorem, we can prove that $\rho \mapsto \tilde{S}_{\eta}\left(C^p_{\rho}\right)$ is continuous on $\left(-1, 1\right)$.

\smallskip
We have $C^p_{1} = C^r$ and $C^p_{-1} = W$. Furthermore, we can show after some algebra that 
\[\tilde{S}_{\eta}\left(C^p_{\rho}\right) \underset{\rho \to 1}{\longrightarrow} 2\Phi\Bigl(\frac{-\eta}{2}\Bigr) = \tilde{S}_{\eta}\left(C^p_{1}\right)\] and \[\tilde{S}_{\eta}\left(C^p_{\rho}\right) \underset{\rho \to -1}{\longrightarrow} \Phi\Bigl(\frac{-\eta}{2}\Bigr) = \tilde{S}_{\eta}\left(C^p_{-1}\right).\] Then $\rho \mapsto \tilde{S}_{\eta}\left(C^p_{\rho}\right)$ is continuous on $\left[-1,1\right]$, which achieves the proof.

\subsection{Proof of Proposition \ref{boundcopuladynamic}}

\noindent (i) As the copulae of $\mathcal{C}^d_G$ are of the form $\left(C_{G,\rho_t}\right)_{t \geq 0}$, the demonstration of this part of the proposition is similar to the one of the static framework.

\smallskip
\noindent (ii) Let $(B^1,B^2)$ be two Brownian motion with copula $C^{ref,\frac{\eta}{2}}$. $B^2$ is then the reflection of $B^1$ according to the stopping time $\tau = \inf \{ t \geq 0 : B^1_t = \frac{\eta}{2}\} = \inf \{ t \geq 0 : B^1_t - B^2_t = \eta\}$. For $t < \tau$, $B^1_t - B^2_t < \eta$ and for $t \geq \tau$, $B^1_t - B^2_t = \eta$. Thus, we have: 
\[S_{\eta, t}\left(C^{ref,\frac{\eta}{2}}\right) = \mathbb{P}_{C^{ref,\frac{\eta}{2}}}\left(t \geq \tau\right) 
= \mathbb{P}_{C^{ref,\frac{\eta}{2}}}\Bigl(\underset{s \leq t}{\sup \,} B^1_s \geq \frac{\eta}{2}\Bigr) 
= 2 \Phi\Bigl(\frac{-\eta}{2\sqrt{t}}\Bigr)
\]according to Lemma \ref{lawsup}.

If $C \in \mathcal{C}_B$, the copula $C_t$ is in $\mathcal{C}$ and then according to Proposition \ref{boundcopula} 
\[
\underset{C \in \mathcal{C}_B}{\sup \;} \mathbb{P}_{C}\left(B^1_t - B^2_t \geq \eta\right) \leq \underset{C \in \mathcal{C}}{\sup \;} \mathbb{P}_{C}\left(B^1_t - B^2_t \geq \eta\right) = 2 \Phi\Bigl(\frac{-\eta}{2\sqrt{t}}\Bigr),\]
which concludes this part of the proof.

\smallskip
\noindent (iii) We want to prove that for all $x$ in $\left[0, 2\Phi\left(\frac{-\eta}{2}\right) \right]$, there exists $C$ in $\mathcal{C}$ such that $\tilde{S}_{\eta}\left(C\right) = x$. Let $x \in \left[0, 2\Phi\left(\frac{-\eta}{2\sqrt{T}}\right) \right]$.

\medskip
If  $x \in \left[0, \Phi\left(\frac{-\eta}{2\sqrt{t}}\right) \right]$, we consider the Gaussian dynamic copula with $\left(\rho_s\right)_{s \geq 0} = 1 - \frac{1}{2t}\left(\frac{\eta}{\Phi^{-1}\left(x\right)}\right)^2 $ which is in $\left[-1,1\right]$ and we have $S_{\eta, t}\left(C_{G,\rho}\right) = x$.

\medskip
If $x \in \left[\Phi\left(\frac{-\eta}{2\sqrt{t}}\right), 2 \Phi\left(\frac{-\eta}{2\sqrt{t}}\right) \right]$, we consider the copula $C^{ref, \frac{\eta}{2} + \lambda}$ defined by Equation \eqref{copulareflectioneq} with $\lambda \geq 0$. With the use of Lemma \ref{lawsup}, we find that
\[\mathbb{P}_{C^{ref, \frac{\eta}{2} + \lambda}}\left(B^1_t - B^2_t \geq x\right)  =
\left\lbrace
\begin{array}{ccc}
\Phi\left(\frac{-x}{2\sqrt{t}}\right) + \Phi\left(\frac{x - 2\eta - 4 \lambda}{2\sqrt{t}}\right) & \mbox{if} & x \leq \eta + 2\lambda \\
0 & \mbox{if} & x > \eta +2\lambda \\
\end{array}\right.
\]
Then, 
\[S_{\eta,t}\left(C^{ref, \frac{\eta}{2} + \lambda}\right) = \Phi\left(\frac{-\eta}{2\sqrt{t}}\right) + \Phi\left(\frac{- \eta - 4 \lambda}{2\sqrt{t}}\right)\]

As we have:
\begin{enumerate}
\item[1)]  $\lambda \mapsto S_{\eta,t}\left(C^{ref, \frac{\eta}{2} + \lambda}\right)$ is continuous on $\left[0,\infty\right)$, 
\item[2)]  $S_{\eta,t}\left(C^{ref, \frac{\eta}{2} + \lambda}\right)  \underset{\lambda \to 0}{\longrightarrow} 2\Phi\left(\frac{-\eta}{2\sqrt{t}}\right)$,
\item[3)]  $S_{\eta,t}\left(C^{ref, \frac{\eta}{2} + \lambda}\right)  \underset{\lambda \to \infty}{\longrightarrow} \Phi\left(\frac{-\eta}{2\sqrt{t}}\right)$,
\end{enumerate}
we can conclude. 

\vip

\noindent {\bf Acknowledgements.} I am grateful to Olivier F\'eron and Marc Hoffmann for helpful discussion and comments. This research is supported by the department OSIRIS (Optimization, SImulation, RIsk and Statistics for Energy Markets) of EDF in the context of a CIFRE contract and by FiME (Finance for Energy Markets) Research Initiative. I thank the referees for valuable comments improving the paper considerably.


\end{document}